\documentclass[12pt,a4paper,twoside]{amsart}
\usepackage{amsmath}
\usepackage{amssymb}
\usepackage{amsfonts}
\usepackage{a4}
\usepackage{amscd}
\numberwithin{equation}{section}

\date{}

\def\ol{overline}
\def\k{\kappa}
\def\suml{\sum\limits}

\def\Ker{\text{\rm Ker}}

\def\bbr{\mathbb{F}}
\def\bbf{\mathbb{R}}
\def\bbp{\mathbb{P}}
\def\bbz{\mathbb{Z}}
\def\bbn{\mathbb{N}}
\def\e{\epsilon}
\def\half{\frac12}

\newtheorem{theorem}{Theorem}[section]
\newtheorem*{theorem*}{Theorem}
\newtheorem{lemma}[theorem]{Lemma}
\newtheorem*{lemma*}{Lemma}
\newtheorem{proposition}[theorem]{Proposition}

\theoremstyle{definition} %fettes label, aber nich kursiv
\newtheorem{definition}[theorem]{Definition}
\newtheorem*{definition*}{Definition}
\newtheorem*{thm1*}{Theorem A1}
\newtheorem*{thm2*}{Theorem A2}
\newtheorem*{claim1*}{Claim 1}
\newtheorem*{claim2*}{Claim 2}
{Corollary}
\newtheorem*{conjecture*}{Conjecture}

\newtheorem*{remark*}{Remarks}
\newtheorem*{remark8.4}{Remark 8.4}
\def\bbz{\mathbb{Z}}
\def\bbq{\mathbb{Q}}
\def\bbf{\mathbb{F}}
\def\bbr{\mathbb{R}}

\def\bbc{\mathbb{C}}

\def\calp{\mathcal{P}}
\def\calo{\mathcal{O}}

\def\calh{\mathcal{H}}
\def\cals{\mathcal{S}}

\def\calr{\mathcal{R}}

\def\ol{\overline}

\theoremstyle{remark}  %kursives label. text roman

\newtheorem*{thmls*}{Theorem (Liebeck-Shalev)}

\begin{document}

\title{Dimension Expanders}
%by \\

\author{Alexander Lubotzky}
\address{
Institute of Mathematics\\
Hebrew University \\
 Jerusalem 91904 ISRAEL}
\email{alexlub@math.huji.ac.il}

\author{ Efim Zelmanov}
\address{
Department of Mathematics\\
University of California at San-Diego (UCSD)\\
La Jolla, CA  92093-0112 USA} \email{ezelmano@math.ucsd.edu}

\thanks{This research was supported by grants from the NSF and the
BSF (US-Israel Binational Science Foundation).}

%\address{
%Institute of Mathematics\\
%Hebrew University\\
%Jerusalem 91904 Israel} \email{alexlub@math.huji.ac.il }

% \qquad A. Lubotzky}
\maketitle

\smallskip \centerline{\emph{ In memory of Walter Feit}}

\smallskip

\baselineskip 13pt
\begin{abstract}
We show that there exists $k \in \bbn$ and $0 < \e \in\bbr$ such
that for every field $F$ of characteristic zero and for every $n
\in \bbn$, there exists explicitly given linear transformations
$T_1,\dots, T_k: F^n \to F^n$ satisfying the following:

For every subspace $W$ of $F^n$ of dimension less or equal $\frac
n2$, $ \dim(W+\suml^k_{i=1} T_iW) \ge (1+\e) \dim W$.  This
answers a question of Avi Wigderson [W].  The case of fields of
positive characteristic (and in particular finite fields) is left
open.
\end{abstract}

\section{Introduction}

A finite $k$-regular graph $X = (V, E)$ with $n$ vertices is
called $\e$-expander (for $0 < \e \in\bbr)$ if for every subset of
vertices $W \subset V$ with $|W|\leq \frac{|V|}{2}, |N(W) | \ge(1+\e) | W|$ where $N(W) = \{ y
\in V| \; \; {\rm distance} (y, W) \le 1\}$.

Most $k$-regular graphs $X$ (when $k$ is even then  all graphs) are
obtained as Schrier graphs of groups (see e.g. [Lu2, (5.4)]) so
there is a finite group $G$ and a symmetric set $S$ of generators
of it of size $k$ acting on $V$ such that the graph structure of
$X$ is obtained by connecting $x\in V$ to $sx, s \in S$.  An
expander graph can be thought of as a permutational representation
of a group $G$, with a set of generators $S=\{s_1,\dots, s_k\}$,
acting on a set $V$ and satisfying for every subset $W$ of size at
most $\frac{|V|}{2}$, $|W \cup\bigcup\limits_{i=1}^k s_i W| \ge 
(1 + \e) | W|.$

Motivated by some considerations from theoretical computer
science, a notion of dimension expanders was suggested by [BISW]
where linear represenations replace permutational
representation.

\begin{definition}
Let $F$ be a field, $k \in \Bbb N$, $\epsilon >0$, $V$ a vector
space of dimension $n$ over $F$ and $T_1, \ldots, T_k$ $F$-linear
transformations from $V$ to $V$.  We say that the pair $(V,
\{T_i\}_{i=1}^{k})$ is an \emph{$\epsilon$-dimension expander} if for every 
subspace  $W$ of dimension less or equal $\frac n2$, $\dim(W+
\suml^k_{i = 1} T_iW) \ge (1+\e)\dim W.$
\end{definition}

It is not difficult to prove that whenever there is a meaningful
measure/probability on $F$ (and hence also on $M_n(F)$ and
$GL_n(F)$) a random choice of $T_1,\dots, T_k$ will give
$\e$-dimension expanders for suitable $k$ and $\e$ (see [Lu1,
Proposition 1.2.1,\break p. 5] for an analogue result for graphs).

In [W], Wigderson posed the problem of
finding, for a fixed field $F$ and for some fixed $k$ and
$\epsilon$, $\epsilon$-dimension expanders of arbitrarily
large dimension.  He suggested there that the set of linear
transformations defined by any irreducible representation
evaluated at expanding generators of the underlying finite
group, gives rise to a dimension expander over the complex
numbers and possibly over finite fields as well. 

Our main result is:

\begin{theorem}
There exist $k \in\bbn$ and $0 < \e \in\bbr$ such that for every
field $F$ of characteristic zero and for every $n$, there are
explicitly given $T_i: F^n \to F^n$, $ i = i, \dots, k$ such that
$(F^n,\{ T_i\}^k_{i = 1})$ is an $\e$-dimension expander.
\end{theorem}

In fact we confirmed Wigderson's suggestion over the complex
numbers but we also show that over finite fields the situation is
more delicate.  To put it on a wider perspective let us recall
that the standard method to get explicit constructions of expander
graphs is by showing that the induced unitary representation of
the group $G$ on the space $\ell^2_0 (V) $ is ``bounded away from
the trivial representation" in the sense of the Fell topology of
the unitary dual of $G$ (see [Lu1, \S 3]).  Here $\ell^2_0(V)$ is
the space of all functions from the set of vertices $V$ to $\bbc$.
We will show in \S 2, that the complex vector space $V=\bbc^n$
becomes a dimension expander with respect to a set of
generators $S$ of $G$, when $G$ acting \emph{unitarily} and
\emph{irreducibly} on $V$, if the adjoint (unitary) representation
of $G$ on $\cals \ell_n(\bbc) = \{ A \in M_n (\bbc)| {\rm trace} A
= 0\}$ is bounded away from the trivial representation.  This
crucial observation enables us to use the known results and
methods developed to construct expander Cayley graphs, to
construct also dimension expanders over $\Bbb C$.  The extensions to
an arbitrary characteristic zero field $F$ is then standard (see
\S3 below).

The use of unitarity is very unfortunate and is the main obstacle
for extending our results to positive characteristic. In fact,
Example 4.IV below implies that there are expander groups whose
linear representations over finite fields are not dimension
expanders.  We still should mention that the finite examples one
may deduce from Example 4.IV are of groups whose order is
divisible by $p$ and the represenations are over $\Bbb F_{p}$.
It may still be that in the ``non-modular'' case, i.e., the
order of the group is prime to the characteristic, Wigderson
suggestion holds even over finite fields.

As of now, the only results over finite fields we are aware of
are of Dvir and Shpilka [DS] who constructed explicit dimension
expanders (of constant expansion) over $GF(2)$ in dimension $n$
with $O(\log n)$ transformations.  They can also expand the
dimension by a smaller factor $1+O(1/ \log n)$ explicitly using
a constant number of transformations. 

Finally, let us relate the results of this note to our work on
``algebras with property $\tau$" [LuZa]: Recall that a group
$\Gamma$ generated by a finite set $S$ is said to have property
$\tau$ if there exists $\e > 0$ such that whenever $\Gamma$ acts
transitively on a finite set $V$ the resulting graph (when $x \in
V$ is connected to $s x, s \in S$) is an $\e$-expander.  The
notion of dimension expanders calls for defining an $F$-algebra
$A$ with a finite set of generators $S$ to have property $\tau$ if
there exists $\e > 0$ such that for every irreducible
representation $\rho$ on a finite dimensional $F$-dimension
$V$, the pair $(V, \rho(S))$ is $\e$-dimension expander.  We
should note however that we do not know in general that for a
group $\Gamma$ with property $\tau$ (or even $T$) the group
algebra $\bbc[\Gamma]$ has $\tau$ (and it looks unlikely - see
examples 4.V and 4.VI). We know this only with respect to
representations of $\bbc[\Gamma]$ which are unitary on $\Gamma$.
So we do not know whether $\bbc[SL_3(\bbz)]$ has $\tau$.  We do
know, though, that $\bbc [SL_3(\bbf_p[t])]$ has $\tau$.  But to
see the delicacy of the issue we show that
$\bbf_p[SL_3(\bbf_p[t])]$ does not have $\tau$, while we do not
know the answer for $\bbf_\ell [SL_3(\bbf_p[t])]$ for a prime
$\ell$ different than $p$.  We elaborate on this in \S 4, leaving
the full treatment to [LuZa] where connections with the
Golod-Shafarevich theory and some questions originated in
3-manifold theory are studied.

We take the opportunity to thank A. Jaikin-Zapirain and A.
Wigderson for useful suggestions.
 
\section{The adjoint representation}

Let $\Gamma$ be a group generated by a finite set $S$ and $(\calh,
\rho: \Gamma \to U(\calh))$ a unitary representation of $\Gamma$.
The Kazhdan constant $K = K^S_\Gamma ((\calh, \rho))$ is defined
as:
\[
\inf\limits_{0 \neq v \in \calh} \mathop{\max}\limits_{s \in S} \;
\left\{ \frac{\| \rho(s) v - v\|}{\| v\|}\right\}
\]

Recall that the group $\Gamma$ is said to have property $T$ if
$K^S_\Gamma = \inf\limits_{(\calh, \rho) \in \calr_0(\Gamma)}
K^S_\Gamma ((\calh, \rho))>  0$ when $\calr_0(\Gamma)$ is the
family of all unitary representations of $\Gamma$ which have no
non-trivial $\Gamma$-fixed vector.  In this case, $K^S_\Gamma$ is
called the Kazhdan constant of $\Gamma$ with respect to $S$.
Similarly $\Gamma$ is said to have property $\tau$ if
$K^S_\Gamma(\tau) = \inf\limits_{(\calh, \rho)
\in\calr^f_0(\Gamma)} K^S_\Gamma(\calh, \rho) > 0$ when
$\calr^f_0(\Gamma)$ is the subset of $\calr_0(\Gamma)$ of all
representations for which $\rho(\Gamma)$ is finite. The number
$K^S_\Gamma(\tau)$ is the $\tau$-constant of $\Gamma$.

If $\calh$ is a finite dimensional space, say $\calh = \bbc^n$ and
$\rho$ a unitary representation $\rho: \Gamma \to U(\calh) =
U_n(\bbc)$, then it induces a representation   $adj \rho$ on
$Hom(\calh, \calh)\simeq M_n(\bbc)$ defined by $adj\rho(\gamma)(T)
= \rho(\gamma)T\rho(\gamma)^{-1}$ for $\gamma \in\Gamma$ and $T
\in M_n(\bbc)$.  The subspace $\cals\ell_n(\bbc)$ of all the
linear transformations (or matrices) of trace $0$ is invariant
under $adj \rho$.  If $\rho$ is irreducible then by Schur's Lemma,
$\cals\ell_n(\bbc)$ does not have any non-trivial $adj
\rho(\Gamma)$-fixed vector.

The space $Hom (\calh, \calh)$ is also a Hilbert space when one
defines for $T_1, T_2 \in Hom (\calh, \calh), \; \; \langle T_1,
T_2\rangle = tr T_1T^*_2$, and $adj \rho$ is a unitary
representation on it and on its invariant subspace
$\cals\ell_n(\bbc)$.

\begin{proposition}
If $\rho:\Gamma \to U_n(\bbc)$ is an irreducible unitary
representation, then $V =\bbc^n$ is an $\e$-dimension expander
for $\rho(S)$ where $\e = \frac{\k^2}{12}$,  $\k =
K^S_\Gamma(\cals\ell_n(\bbc), adj\rho)$.
\end{proposition}
\begin{proof}
Let $W \le V$ be a subspace of dimension $m \le \frac n2$.  Let
$P$ be the linear projection from $V$ to $W$.  As $P^* = P$, we
have that $\langle P, P\rangle = tr(P^2) = tr (P) = m$.  The right
hand equality is seen simply by considering a basis which is a
union of a basis of $W$ and of $W^\perp$.  Consider $Q = P - \frac
mn I$ where $I$ is the identity operator.  Then
\[
tr Q = tr P - tr(\frac mn I) = 0,\quad \,  {\rm so} \; \; Q \in
\cals\ell_n(\bbc).
\]
The norm of $Q$ is:
\begin{align*}
&\| Q\|^2= tr ((P - \frac mn I)(P-\frac mn I)^*) = \\
=&tr (P^2 - 2\frac mn P + \frac{m^2}{n^2}I) = m - 2 \frac{m^2}{n}
+ \frac{m^2}{n^2} n = m-\frac{m^2}{n}.
\end{align*}

It follows now that there exists $s\in S$ such that $\| \rho(s)
(Q) - Q\|^2 \ge \kappa^2\|Q\|^2$.  Notice that $\rho(s)(Q) =
\rho(s)Q\rho(s)^{-1} = \rho(s) P\rho(s)^{-1} - \frac mn I = P' -
\frac{tr P'}{n} I$ where $P' = \rho(s)P\rho(s)^{-1}$ is the
projection of $V$ onto the subspace $W' = \rho(s)(W)$, so $tr P' =
m$ as well.
%\end{proof}
\begin{lemma}
If $W, W'$ are two subspaces of $V$ of dimension $m$ and $P, P'$ the 
projections to $W$ and $W'$, respectively, then $Re<P, P'> \geq
4m-3 dim(W+W')$. \end{lemma}

\begin{proof}
Denote $U_0 = W\cap W', U_1 = U^\perp_0 \cap W,$ and $U_2 =
U^\perp_0
\cap W'$ and $d_i = dim U_i$.  Then $V=U_0 \oplus U_1 \oplus U_2
\oplus(W+W')^{\perp}$.

Let us choose an orthonormal basis for $V$ compose of
$\{\alpha_1, \ldots, \alpha_{d_{0}}\}$ an orthonormal basis for
$U_0, \{\beta_1 , \ldots, \beta_{d_{1}+d_{2}}\}$ an orthonormal
basis for $U_1 \oplus U_2$ and $\{\gamma_1 ,  \ldots,
\gamma_r\}$, $r=dim V-(d_0 +d_1 +d_2)$, an orthonormal basis for
$(W+W')^{\perp}$.  The operator $PP'$ is the identity on $U_0$
and zero on $(W+W')^{\perp}$.  It is a \break linear transformation of
$norm \leq 1$, so $|<PP'\beta_i , \beta_i >|\leq 1$ and hence: 

$$
-1 \leq Re <PP' \beta ,\beta_i> \leq 1. \leqno(*)
$$

The trace of $PP'$ on $U_0$ is $dim U_0$ and it is 0 on
$(W+W')^{\perp}$.  Together with (*) we get:

$$
Re \ tr(PP') \geq dim U_0 - dim U_1 - dim U_2 =
$$

$$
=dim(W \cap W')-(dim V - dim(W \cap W')-dim (W+W')^{\perp})=
$$

$$
=2 dim(W \cap W')-dim(W+W')=
$$

$$
=4m-3dim(W+W')
$$

The last equality follows from the fact that 
$$
dim(W\cap
W')=dimW+dimW'-dim(W+W')=
$$

$$
=2m-dim(W+W').
$$
\end{proof}

\smallskip

So altogether,
\begin{align*}
\k^2 (m-\frac{m^2}{n}) &\le\, \| \rho(s)(Q) - Q\|^2 = \| \rho(s)
P\rho(s^{-1}) - P\|^2 = \| P'-P\|^2\\
&=\, \langle P'-P,\, P'-P\rangle = \langle P', P'\rangle +\langle
P, P\rangle  - \langle
P, P'\rangle - \langle P', P\rangle \\
&=\, 2m-2Re<P,P'> \\
&\, \leq 2m-2(4m-3dim(W+W'))=6dim(W+W')-6m,.
\end{align*}

and therefore

$$
m\Big(1+\frac{\kappa^{2}}{6}\Big(1-\frac{m}{n}\Big)\Big) \leq dim(W+W').
$$

As $1-\frac{m}{n} \geq \frac{1}{2}$, we get that $ dim W(1+
\frac{\kappa^{2}}{12}) \leq dim(W+W')$ and Proposition 2.1 is now 
proved.
\end{proof}

\medskip

\noindent{\bf Remark/Question 2.5} \ \  For graph expanders or
equivalently, as in the introduction, for permutational
representation of $\Gamma$ on a set $V$, $V$ is an expander iff
the representation of $\Gamma $ on $\ell^2_0 (V)$ is bounded away
from the trivial representation (cf. [Lu1,  4.3]). Do we have such
a converse for Proposition 2.1?  I.e., if a $\bbc$-vector space
$V=\bbc^n$ is a dimension expander with respect to a unitary
representation $\rho$ of $\Gamma$, does it imply that the
representation $adj \rho$ on $\cals \ell_n(\bbc)$ is bounded away
from the trivial representation?

\section{Examples and a proof of Theorem 1.2}

There are many known examples of groups with property $T$ or
$\tau$. They can now, in light of Proposition 2.1, be used to
give a proof for Theorem 1.2, i.e. to give explicit sets of linear
transformations of $\bbc^n$ which  solve  Wigderson's Problem.

Let us take some of the examples which are the simplest to
present:

\medskip
\noindent{\bf 3.I} \ \  Fix $3 \le d \in \bbn$ and let $\Gamma =
SL_d(\bbz)$ with a fixed set of generators, e.g. $S = \{ A, B \}$
when $A$ is the transformation sending $e_1$ to $e_1 + e_2$ and
fixing $e_2,\dots, e_d$ and $B$ will send $e_i\to e_{i + 1}$ for
$i = 1,\dots, d-1$ and $e_d $ to $(-1)^{d-1} e_1$.  (Here
$\{e_1,\dots, e_d\}$ is the standard basis for $\bbz^d$).

As $\Gamma$ has Property $T$, whenever we take a finite
dimensional irreducible unitary representation $\rho$ of $\Gamma $
on $V=\bbc^n$, \  $\rho(A)$ and $\rho(B)$ make $V$ a dimension
expander.

\medskip

\noindent{\bf Remark/Question 3.1} \ \ Is this true also for the
non-unitary representations of $\Gamma$?  Note that, say
$SL_3(\bbz)$, has infinitely many irreducible rational
representations (i.e. rational representations of $SL_3(\bbc)$
restricted to the Zariski dense subgroup $SL_3(\bbc)$).  These
representations are classified by the highest weights of
$\cals\ell_3(\bbc)$.  Are they dimension expanders with respect
to $\rho(A)$ and $\rho(B)$?

\bigskip
\noindent {\bf 3.II} \ \  Fix $3\le d \in\bbn$ and a prime $p$ and
let $\Gamma = SL_d (\bbf_p[t])$.  This $\Gamma $ has $(T)$ and all
its representations over $\bbc$ factor through finite quotients
([Ma, Theorem 3, p.3]) which implies that they can be unitarized.
Thus unlike in example 3.I, we can deduce that all the irreducible
complex representations of $\Gamma$ give rise to dimension expanders.

\bigskip
\noindent {\bf 3.III} \ \  Let $\Gamma = SL_2(\bbz)$ and for a
prime $p$, denote $\Gamma (p) = \Ker (SL_2(\bbz) \to
SL_2(\bbz/p\bbz))$.  Let $A = \begin{pmatrix} 1 &1\\0
&1\end{pmatrix} $ and $B =\begin{pmatrix} 0 &1\\-1
&0\end{pmatrix}$ form a set of generators $S$ for $\Gamma$.  The
group $\Gamma$ has property $\tau$ with respect to the family $\{
\Gamma (p)\}$ -- see [Lu1, 4].

This means that there exists $\k>0$ such that for all unitary
representations $(V,\rho)$ of $\Gamma$ which factor through
$\Gamma /\Gamma (p) = SL_2(p)$ for some $p$ and do not have a
fixed point, $K^S_\Gamma (V,\rho) > \k$.  The group $SL_2 (p)$
acts on the $\bbf_p$-projective line $\bbp^1=\{0, \dots, p-1,
\infty\}$ via $\begin{pmatrix} a &b\\c &d\end{pmatrix} (z) =
\frac{az+b}{cz+d}$. This is a double transitive permutational
representations (and indeed give rise to expander graphs -- see
[Lu1, Theorem 4.4.2]). Moreover, it induces an irreducible linear
representation $\rho$ on $\ell^2_0 (\calp^1)\cong \bbc^p$.  The
adjoint representation of $\Gamma$ on $\cals\ell_p(\bbc)$ also
factors through $SL_2(p)$ and is therefore bounded away from the
trivial representation.  Hence $\bbc^p$ are dimension expanders
with respect to the explicitly given transformations $\rho(A) $
and $\rho(B)$.

\bigskip
\noindent{\bf 3.IV} \ \ \ Recently Kassabov ([K2]) showed that the
Cayley graphs of the symmetric groups $\Sigma_n$ are expanders
with respect to an explicitly given set $S_n$ of generators with
$|S_n| \le k $ for some $k$ ($k\le 30$ in his work).  This means
that there exists $\k > 0$ such that $K^{S_n}_{\Sigma_n} (V, \rho)
>\k$ for every $n$ and every representation $(V, \rho)
\in \calr_0 (\Sigma_n)$. This can be applied in particular to $adj
\rho_n$ when $\rho_n$ is the linear representation of $\Sigma_n$
on $\bbc^{n-1}\cong \ell^2_0 (\{ 1,\dots, n\})$ induced from the
natural permutational representation of $\Sigma_n$ on the set $\{
1, \dots, n\}$.  We  get this way that $\bbc^{n-1}$ are
dimension expanders for every $n$ with respect to $k$ explicit
linear transformations.

\medskip

Examples 3.III and 3.IV are especially useful for us as the
representations $\rho$ which appear in these examples are all
defined over $\bbq$. This enables us now to prove Theorem 1.2 over
every field of characteristic zero. Indeed, if $F$ is such a field
then $F$ contains $\bbq$ and so the examples 3.III and 3.IV give
finite representations $\rho$ into $GL_n (F)$ for various values
of $n$. Now if $|F|\le \aleph$, then $F$ can be embedded into
$\bbc$ and so $GL_n(F) \subset GL_n (\bbc)$ and as $\rho$ are
finite, they can be unitarized over $\bbc$.  Now, as $\bbc^n =
\bbc\mathop{\otimes}\limits_F F^n$, every $F$-subspace $W$ of
$F^n$ spans a $\bbc$-subspace $\ol W$ of $\bbc^n$ of the same
dimension. It is easy to see that for $\rho(s) \in GL_n(F)$,
$\dim_F (\rho(s) W+W) = \dim_\bbc (\rho(s)\ol W + \ol W)$ hence
$F^n$ are also dimension expanders.

This actually covers also the general case: if $F$ has  large
cardinality and $W\le F^n$ is a counter example, then the entries
of a basis of $W$ generate a finitely generated field $F_1$ and
the counterexample can already be found in $F^m_1$. This is
impossible by the previous paragraph - so Theorem 1.2 is now
proved for every field $F$ or characteristic zero.

\medskip

The above discussion may give the impression that the only way to
make $\bbc^n$ into dimension expander may be via unitary and
even finite representation.  This is not the case.  Let us
consider more examples:

\medskip \noindent{\bf 3.V} \ \  Fix $5\le d \in \bbn$ and $p$ a
prime. Let $f$ be the quadratic form $x^2_1 + \dots +x^2_{d-2} -
\sqrt{p}\; x^2_{d-1} -\sqrt{p}\;  x^2_d $ and let $\Gamma $ be the group
 $SO(\calo, f)$ of all the $d\times d$ matrices over $\calo$ of
 determinant $1$ preserving $f$, where $\calo$ is the ring of
 integers of the field $E=\bbq(\sqrt{p})$. The group $\Gamma$ is a 
Zariski dense subgroup of the $E$-algrebraic
group $H=SO(f)$.

Let $G=Res_{E/Q}(H)$ be the restriction of scalors from $E$ to
$\Bbb Q$ of $H$.  Then $\Gamma$ sits diagonally as a lattice the
real points of $G$ which is isomorphic to $SO(\Bbb R, f) \times SO(\Bbb
R, f^\iota)=SO(d-2,2) \times SO(d)$ (([Ma, Theorem 3.2.4])) and it
therefore has property $T$.  Here, $\iota$ is the unique non-trivial 
element of the Galois group $Gal(E/\Bbb Q)$.  

Every irreducible $E$-rational represenation
of $H$ define two real representations of $\Gamma$-one for each
of the two embeddings of $E$ into $\Bbb R$.  For an element $g$
of $\Gamma$ if one represenation sends it to a matrix $A$, the second
sends it to $A^{\iota}$ - i.e., applying $\iota$ to all the entries of
$A$.  It follows, that if one representation define
$\epsilon$-dimension expander so is the other.  Now, one of
these two representation factors through $SO(d)$ and hence it is
always unitary.  As $\Gamma$ has $T$, it follows that it defines
a dimension expander.  The other one is never unitary as its
real Zariski closer is $SO(d-2,2)$.  Still, we deduce that it is
a dimension expander.  By taking infinitely many such
$E$-rational represenation one sees that neither finiteness nor
unitarity are necessary for dimension expanders.         

It is our lack of other methods which forces us to use unitarity
(even in a non-unitary example).  It will be extremely
interesting and important to develop methods to build dimension
expanders not via unitarity.  This is exactly the obstacle that
should be overcome in order to answer Wigderson's problem for
fields of positive characteristic.

\section{Algebras with property $\tau$}

In this section we briefly comment on the connection between
dimension expanders and algebras with property $\tau$.  A
more systematic study of these algebras will be given in [LuZa].

Motivated by the definition of groups with $\tau$ and by some
problems on Golod-Shafarevich groups and 3-manifold groups, we
would like to have a notion of algebras with property $\tau$.
There are several reasonable options.  The most natural one is
probably:
\begin{definition}  Let $F$ be a field and $A$ an $F$-algebra
generated by a finite set $S$.  $A$ is said to have \emph{property
$\tau$} if there exists an $\e > 0$ such that for every
$(F-)$finite dimensional simple $A$-module $V$ and for every
subspace $W$ of $V$ of dimension most $\half \dim V,\,
\dim(W+\suml_{s\in S} s W) \ge (1+\e) \dim W$.  Namely, $V$ is an
$\e$-dimension expander with respect to $S$.
\end{definition}

It is not difficult to see that if $A$ has $\tau$ with respect to
$S$, it has it with respect to any other finite set of generators,
possibly with a different $\e$.

\bigskip
\noindent{\bf Examples}:

\bigskip

\noindent{\bf 4.I.} \ Let $A$ be the free algebra over $F$ on one
generator $x$, i.e. $A = F[x]$.  Assume $F$ has field extensions
of infinitely many different degrees, i.e., $\{ F_i\}_{i\in\bbn}$ are
finite field extensions with $n_i = [F_i: F]\ge i$.  So, $F_i =
F[x]/(f_i(x))$, when $f_i(x)$ is an irreducible polynomial of
degree $n_i$, is an irreducible $A$ module in a clear way. Let $W$
be the span of $x + (f_i(x)),\dots, x^{[\frac{n_i}{2}]-1} +
(f_i(x))$ in $F_i$.  Then $\dim (W+xW) \le \left(
1+\frac{2}{n_i}\right)\dim W$ and hence $A$ does not have property
$\tau$.

\bigskip

\noindent{\bf 4.II}.  \ \ On the other hand if $F$ is
algebraically closed, then every finite dimensional simple module
of $F[x]$ is one dimensional and hence $F[x]$ has $\tau$ in a
vacuous way.  (This is perhaps a hint that Definition 4.1 is maybe
not the most appropriate one).  But see also Example 4.VI below
for comparison.

\bigskip

\noindent{\bf 4.III} \ \  If $F$ as in 4.I and $d\in\bbn$ a fixed
integer, then by a similar argument $M_d(F[x])$ does not have
property $\tau$.

\bigskip

\noindent{\bf 4.IV} \ \  Let $p$ be a fixed prime, $3\le d
\in\bbn$ a fixed integer and $\Gamma = SL_d(\bbf_p[x])$.  It
follows from Example 3.II that the group algebra $\bbc[\Gamma]$
has property $\tau$ (and,  in fact, it follows that $F[\Gamma]$
has $\tau$ for every characteristic zero field.)  On the other
hand the algebra $\bbf_p[SL_d(\bbf_p[x])]$ is mapped
epimorphically onto $M_d(\bbf_p[x])$ via the natural embedding
$SL_d(\bbf_p[x] ) \subseteq M_d(\bbf_p[x])$.  The latter does not
have $\tau$ by 4.III.  Hence $\bbf_p[SL_d(\bbf_p[x])]$ does not
have $\tau$. We do not know if for a prime $\ell \neq p$,
$\bbf_\ell[SL_d(\bbf_p[x])]$ has $\tau$ or not.

\bigskip

\noindent{\bf 4.V} \ \  It was shown recently by Kassabov and
Nikolov [KN] that for $d \ge 3$, the group $\Gamma =
SL_d(\bbz[x])$ has property $\tau$.  Since $SL_d(\bbz[x])$ is
mapped onto $SL_d(\bbf_p[x])$ for every $x$, it follows that for
every $p, \bbf_p[\Gamma]$ does not have $\tau$.  We do not know if
$\bbc[SL_d(\bbz[x])]$ has $\tau$.  If it has, the same would apply
for $\bbc[SL_d(\bbz)]$ which would resolve Question 3.1.

\bigskip

\noindent{\bf 4.VI} \ \  Let $R = \bbz \langle x, y\rangle$ be the
free non-commutative ring on $x$  any $y$.  Fix $3 \le d \in\bbz$
and let $\Gamma = EL_d(R)$ -  the group generated by the
elementary unipotent matrices $ I + rE_{ij}$ inside $M_d(R)$, when
$r \in R$ and $1 \le i \neq j \le d$.  It was conjectured by
Kassabov [K1] that $\Gamma$ has $\tau$ and some partial results in
this direction are proved there . We claim that for every field
$F$, $F[\Gamma]$ - the group algebra of $\Gamma$ over $F$, does
not have property $\tau$. To see this, let us observe first that
$F\langle x, y\rangle = F\otimes_\bbz \bbz \langle x, y \rangle$
does not have $\tau$. Indeed, the latter would have implied that
there exists $\e > 0$ such that for any $n \in \bbn$ and any two
generators $a$ and $b$ of $M_n(F)$ and any subspace $W$ of $F^n$
with $\dim W\le \frac n2$, we have $\dim (W + a W +bW)\ge (1+\e)
\dim W$, which is clearly not the case.
 A general (easy)
argument shows that if an $F$-algebra $A$ does not have $\tau$
then $M_d(A)$ does not have it either (see [LuZa] for this and
more). Thus $M_d(F\langle x, y\rangle)$ does not have $\tau$.  The
algebra $F[\Gamma]$ is mapped onto $M_d (F\langle x, y\rangle )$
and so the same applies to it.

It is interesting to compare the last example with the results of
Elek ([E1], [E2]).  He defined  a notion of ``amenable algebra"
and proved that a group $\Gamma$ is amenable if and only if its
group algebra $\bbc[\Gamma]$ is amenable.

\medskip

The above examples illustrate the delicacy of finding algebras
with property $\tau$ over finite fields or similarly dimension
expanders.  The lack of unitarity is the main obstacle and the
main problem we leave open is to find a method replacing it for
fields of positive characteristic in general and finite fields in
particular.


\begin{thebibliography}{cc}

\bibitem[BISW]{BISW} B. Barak, R. Impagliazzo, A. Shpilka and A.
Wigderson, \emph{Private communication}.

\bibitem[DS]{DS} Z. Dvir, A. Shpilka, In prepartion.
 
\bibitem[E1]{E1} G. Elek, \emph{The amenability of affine algebras}. J. Algebra {\bf 264}(2003),
469--478.

\bibitem[E2]{E2} G. Elek, \emph{The amenability and
non-amenability of skew fields}. arXiv:math.RA/0311375

\bibitem[K1]{K1} M. Kassabov,  \emph{Universal lattices and
unbounded rank expanders}, preprint February 2005.
arXiv:math.GR/0502237



\bibitem[K2]{K2} M. Kassabov, \emph{Symmetric groups and expanders} (announcement), Electronic
 Research Announcements of AMS (to appear). arXiv:math.GR/0503204




\bibitem[KN]{KN} M. Kassabov, N. Nikolov, \emph{Universal lattices
and property $\tau$}. Preprint, February 2005.
arXiv:math.GR/0502112


\bibitem[Lu1]{Lu1} A. Lubotzky, \emph{Discrete groups, expanding graphs and invariant measures}.
Progress in Mathematics, 125. Birkhuser Verlag, Basel, 1994.
xii+195 pp

\bibitem[Lu2]{Lu2} A. Lubotzky, \emph{Cayley graphs: eigenvalues, expanders and random walks}.
 Surveys in combinatorics, 1995 (Stirling), 155--189, London Math. Soc.
 Lecture Note Ser., 218, Cambridge Univ. Press, Cambridge, 1995.

\bibitem[LuZa]{LuZa} A. Lubotzky, E. Zelmanov, \emph{On algebras
with property $\tau$}. In preparation.

\bibitem[Ma]{Ma} G. Margulis, \emph{Discrete subgroups of semisimple Lie groups}.
 Ergebnisse der Mathematik und ihrer Grenzgebiete (3) 17.
 Springer-Verlag, Berlin, 1991. x+388 pp.

\bibitem[W]{W} A. Wigderson, A lecture at \emph{IPAM, UCLA},
February 2004.




\end{thebibliography}
\end{document}